%% file: Main.tex
\documentclass[12pt]{amsart}

\usepackage{amsmath}
\usepackage{amssymb}
\usepackage{amsfonts}
\usepackage{bbm}
\usepackage{mathrsfs}
\usepackage{array}

\usepackage[usenames,dvipsnames]{color}
\usepackage{hyperref}
\usepackage{cite}

\newtheorem{definition}{Definition}
\newtheorem{lemma}{Lemma}
\newtheorem{theorem}{Theorem}
\newtheorem{corollary}{Corollary}
\theoremstyle{definition}
\newtheorem{remark}{Remark}

\DeclareMathOperator{\HyperF}{F}

\newcommand{\Hypergeom}[5]{{\sideset{_#1}{_#2}\HyperF\!\left(
      \genfrac{}{}{0pt}{0}{#3}{#4}
      \middle|\,#5\right)}}

\DeclareMathOperator{\csing}{\mathrm{C}_{\mathrm{sing}}}
\DeclareMathOperator{\cont}{\mathrm{C}}
\newcommand{\eps}{\varepsilon}

\renewcommand{\theta}{\vartheta}
\newcommand{\Laplace}{\triangle}

\begin{document}
\title[Positive definite singular kernels\ldots]%
{Positive definite singular kernels on two-point homogeneous spaces}

\author[Dmitriy Bilyk]{Dmitriy Bilyk\textsuperscript{*}}
\thanks{\textsuperscript{*} This author was supported by the National Science
  Foundation grant DMS 2054606 and is also grateful to Fulbright Austria and
  NAWI Graz.} 
\email{dbilyk@umn.edu}
\address{School of Mathematics,
  University of Minnesota,
  Minneapolis, MN 55455, USA}

\author[Peter J. Grabner]{Peter J. Grabner\textsuperscript{\ddag}}
\thanks{\textsuperscript{\ddag} This author is supported by the Austrian
  Science Fund FWF project I~6750}
\email{peter.grabner@tugraz.at}
\address{Institute of Analysis and Number Theory,
  Graz University of Technology,
  Kopernikusgasse 24.
8010 Graz,
Austria}
\dedicatory{Dedicated to Edward B. Saff on the occasion of his
  80\textsuperscript{th} birthday}
\maketitle

\begin{abstract}
  We study positive definiteness of kernels $K(x,y)$ on two-point homogeneous
  spaces. As opposed to the classical case, which has been developed and studied
  in the existing literature, we allow the kernel to have an (integrable)
  singularity for $x=y$. Specifically, the Riesz kernel $d(x,y)^{-s}$ (where
  $d$ denotes some distance on the space) is a prominent example. We derive
  results analogous to Schoenberg's characterization of positive definite
  functions on the sphere, Schur's lemma on the positive definiteness of the
  product of positive definite functions, and Schoenberg's characterization of
  functions positive definite on all spheres. We use these results to better
  understand the behavior of the Riesz kernels for the geodesic and chordal
  distances on projective spaces.
\end{abstract}

\input{Introduction}
\input{Harmonic}
\input{PosDef}
\input{Schur}
\input{Schoenberg}
\input{Riesz}

\bibliographystyle{amsplain}
\bibliography{refs}
\end{document}

%% file: Introduction.tex
\section{Introduction}\label{sec:introduction}
In the present paper we explore the connection between positive definiteness,
energy minimization, and positivity of coefficients in the Jacobi expansions
for singular kernels on two-point homogeneous spaces. While vast literature on
this topic exists in the case of continuous or bounded kernels
\cite{Schoenberg1942:positive_definite_functions,
  Borodachov_Hardin_Saff2019:discrete_energy_rectifiable}, the general case of
singular kernels is much less studied despite the ubiquity of such kernels --
the most common example are the Riesz kernels with respect to various metrics,
which arise naturally in harmonic analysis, potential theory, mathematical
physics etc. Specific singular Riesz kernels of this type have been
investigated in this context in numerous papers
\cite{Anderson_Dostert_Grabner+2023:riesz_green_energy,
  Damelin_Grabner2003:energy_functional_numeric,
  Bilyk_Dai2019:geodesic_distance_riesz}, but all of the prior
results relied on ad hoc approximation tricks and a universal theory was still
lacking (an attempt at a general study of energy minimization for singular
kernels has been undertaken in
\cite{Bilyk_Matzke_Nathe2024:geodesic_distance_riesz}). The results presented
in the current paper are stronger and broader. We shall demonstrate that most
results about continuous positive definite kernels propagate to singular
kernels under minimal integrability (and continuity away from the singularity)
conditions and without any structural assumptions about the kernel.

In the following we will restrict ourselves to compact two-point homogeneous
spaces. These are metric spaces $(X,d)$ with an isometry group $G$ acting on
them such that for pairs of points $(x_1,y_1)$ and $(x_2,y_2)$ with
$d(x_1,y_1)=d(x_2,y_2)$ there exists a $g\in G$ such that $gx_1=x_2$ and
$gy_1=y_2$. The two-point homogeneous compact manifolds have been classified
(see \cite[Chapter I.4]{Helgason2000:groups_geometric_analysis}) to be the
projective spaces
$\mathbb{RP}^{d-1},\mathbb{CP}^{d-1},\mathbb{HP}^{d-1},\mathbb{OP}^2$, and the
spheres $\mathbb{S}^{d-1}$; in general, we write $\mathbb{FP}^{d-1}$ for the
projective spaces over the ground ``field'' $\mathbb{F}$. We collect all the
necessary facts about the harmonic analysis on these spaces in Section
\ref{sec:harm-analys-two}.

While most of our results are rather general, we shall pay special attention to
Riesz kernels with respect to two special distances on these spaces: the
chordal distance (which is an analogue of the Euclidean distance on the sphere)
and the geodesic distance intrinsic to the space. Even though powers of these
distances are equivalent at small scales, they behave quite differently with
respect to positive definiteness and energy minimization, which has already
been observed in \cite{Gangolli1967:positive_definite_kernels}. This difference
is particularly striking in the case of the projective spaces, and we further
investigate and explain this curious phenomenon.

The outline of the paper is as follows. In Section~\ref{sec:harm-analys-two} we
present the necessary background on the harmonic analysis on two-point
homogeneous spaces (zonal kernels, Jacobi polynomials, Poisson kernels).
Section~\ref{sec:pos-def} explores the positive definiteness property in the
case of singular kernels. In particular, in Theorems~\ref{thm:posdef}
and~\ref{thm:strict} we extend the classical results that go back to Schoenberg
\cite{Schoenberg1942:positive_definite_functions} (for projective spaces, see
e.g. \cite{Gangolli1967:positive_definite_kernels}) to the case of singular
zonal kernels, i.e. we show that a singular zonal kernel is (strictly) positive
definite if and only if all of its Jacobi coefficients are non-negative
(positive) when only integrability with respect to the uniform (i.e. isometry
invariant) measure and continuity away from the singularity are assumed. In
Section~\ref{sec:schurs-lemma} we demonstrate that the classical Schur's lemma
generalizes to singular kernels: the product of positive definite kernels is
again positive definite (as long as this product is integrable). In
Section~\ref{sec:functions-for-all} we extend Schoenberg's
\cite{Schoenberg1942:positive_definite_functions} results on functions, which
are positive definite on spheres of all dimensions to the case of projective
spaces. For continuous functions such results have been obtained by
\cite{Guella_Menegatto2018:limit_formula_jacobi}, we extend these to the case
of singular integrable kernels. As a byproduct we obtain that no geodesic Riesz
kernel can be positive definite on all projective spaces.


%% file: Harmonic.tex
\section{Harmonic analysis on two-point homogeneous
  spaces}\label{sec:harm-analys-two}
In this section we collect some basic facts about the harmonic analysis on the
spheres $\mathbb{S}^{d-1}$ and the projective spaces $\mathbb{FP}^{d-1}$. It is
convenient to set $\frac{\pi}{2\kappa}$ as the geodesic diameter of these
spaces, where $\kappa=\frac12$ for the spheres and $\kappa=1$ for the
projective spaces.  We denote the geodesic distance by $\theta(x,y)$.
Furthermore, we set
\begin{equation}
  \label{eq:alpha-beta}
  \alpha=\frac{d-1}2\dim_{\mathbb{R}}(\mathbb{F})-1\quad\text{and }
  \beta=
  \begin{cases}
    \alpha&\text{for }X=\mathbb{S}^{d-1}\\
    \frac{\dim_{\mathbb{R}}(\mathbb{F})}2-1&\text{for }X=\mathbb{FP}^{d-1}.
  \end{cases}
\end{equation}
Throughout this paper we will denote the space associated to parameters
$\alpha$ and $\beta$ by $X_{\alpha,\beta}$ (but will sometimes drop the indices
if it does not cause confusion). Furthermore, we denote by
$D=2\alpha+2=(d-1)\dim_{\mathbb{R}}$ the (real) dimension of the space.

It then follows that the surface area of the geodesic sphere of radius
$r\in[0,\frac{\pi}{2\kappa}]$ is given by (see
\cite[Proposition~5.6]{Helgason1965:Radon_Trans_Two_Pt_Homo})
\begin{equation}
  \label{eq:geo-sphere}
  A(r)=C_{\alpha,\beta}\kappa\sin(\kappa r)^{2\alpha+1}\cos(\kappa r)^{2\beta+1},
\end{equation}
where the constant is chosen as
\begin{equation*}
 C_{\alpha,\beta}=\frac{2\Gamma(\alpha+\beta+2)}{\Gamma(\alpha+1)\Gamma(\beta+1)}
\end{equation*}
to normalize the volume of the space
\begin{equation*}
  \int_0^{\frac\pi{2\kappa}}A(r)\,dr .
\end{equation*}
The Laplace operator for functions only depending on the distance to one fixed
point (``zonal functions'') is given by (see
\cite[Chapter~X.7.4]{Helgason1978:Diff_Geo_Lie_Group_Sym_Spaces})
\begin{equation}
  \label{eq:laplace}
  \Laplace f=-\frac1{A(r)}\frac{\partial}{\partial r}
  \left(\frac{\partial}{\partial r}A(r)f\right).
\end{equation}
This implies that the zonal eigenfunctions of the Laplace operator are given by
\begin{equation}
  \label{eq:laplace-eigen}
  P_n^{(\alpha,\beta)}(\cos(2\kappa r)),
\end{equation}
where $P_n^{(\alpha,\beta)}$ denotes the Jacobi polynomial. The corresponding
eigenvalues are then
\begin{equation*}
  \lambda_n=4\kappa^2n(n+\alpha+\beta+1),
\end{equation*}
see \cite[Theorem~4.2.2]{Szegoe1975:orthogonal_polynomials}.

Let $\sigma_{\alpha,\beta}$ denote the normalized surface measure on
$X_{\alpha,\beta}$, which is induced by the Haar measure on the group $G$. Then
the equality
\begin{align}
  \label{eq:zonal-int}
  \int_{X_{\alpha,\beta}} F(\cos(2\kappa\theta(x,a)))\,d\sigma_{\alpha,\beta}(x)& =
  \int_0^{\frac\pi{2\kappa}}F(\cos(2\kappa\theta))\,d\nu_{\alpha,\beta}(\theta)\\
  \nonumber & = \int_{-1}^1 F(t) \, d\mu_{\alpha,\beta} (t),
\end{align}
where
\begin{align*}
  d\nu_{\alpha,\beta}(\theta)&=C_{\alpha,\beta}\kappa\sin(\kappa\theta)^{2\alpha+1}
  \cos(\kappa\theta)^{2\beta+1}\,d\theta,\\
  d\mu_{\alpha,\beta}(\theta)&=C_{\alpha,\beta} (1-t)^\alpha (1+t)^\beta dt,
\end{align*}
holds for every function $F:[-1,1]\to\mathbb{R}$, which is integrable with
respect to $d\mu_{\alpha,\beta}$, or equivalently,  $F(\cos2\kappa\theta ) \in L^1 (\nu_{\alpha,\beta})$.

Since the Jacobi polynomials are orthogonal with respect to
$d\mu_{\alpha,\beta}$, every function $F\in L^2(\mu_{\alpha,\beta})$ can be
represented as a ``Fourier series''
\begin{equation}\label{eq:Fourier}
  F(t)=
  \sum_{n=0}^\infty\widehat{F}(n)P_n^{(\alpha,\beta)}(t),
\end{equation}
where
\begin{equation}\label{eq:defcoef}
 \widehat{F}(n)=\frac{m_n^{(\alpha,\beta)}}{\left(P_n^{(\alpha,\beta)}(1)\right)^2}
  \int_0^{\frac\pi{2\kappa}}F(\cos(2\kappa\theta))
  P_n^{(\alpha,\beta)}(\cos(2\kappa\theta))
  \,d\nu_{\alpha,\beta}(\theta)
\end{equation}
and
\begin{equation}\label{eq:mn}
  m_n^{(\alpha,\beta)}=\frac{2n+\alpha+\beta+1}{\alpha+\beta+1}
  \frac{(\alpha+\beta+1)_n(\alpha+1)_n}{n!(\beta+1)_n}.
\end{equation}

Equality \eqref{eq:Fourier} is \emph{a priori} only valid in the
$L^2$-sense. For instance, if $F$ is continuous and all coefficients
$\widehat{F}(n)$ are non-negative, then the series converges absolutely and
uniformly (which can be viewed as a consequence of Mercer's theorem, see
\cite{Mercer1909:functions_positive_negative}).

For the values $\alpha$ and $\beta$ occurring in our context, the
numbers $m_n^{(\alpha,\beta)}$ are integers and equal the dimension of the space of
eigenfunctions of the Laplace operator \eqref{eq:laplace} for the eigenvalue
$\lambda_n$.

We define the following function on $[-1,1]$ for $0\leq r\leq1$
\begin{equation}\label{eq:poisson}
  \begin{split}
    \mathcal{P}_r(\cos2\kappa\theta)&= \sum_{n=0}^\infty
    \frac{2n+\alpha+\beta+1}{\alpha+\beta+1}
    \frac{(\alpha+\beta+1)_n}{(\beta+1)_n}
    P_n^{(\alpha,\beta)}(\cos2\kappa\theta)r^n\\
    &=    \frac{1-r}{(1+r)^{\alpha+\beta+2}}
    \Hypergeom21{\frac{\alpha+\beta+2}2,\frac{\alpha+\beta+3}2}{\beta+1}%
    {\frac{4r\cos^2\kappa\theta}{(1+r)^2}},
  \end{split}
\end{equation}
which we will call the \emph{Poisson kernel}. Here and throughout $_2F_1$
denotes the classical hypergeometric function (see, for instance
\cite{Andrews_Askey_Roy1999:special_functions}). For the special case
$X=\mathbb{S}^{d-1}$ this is the classical Poisson kernel, which solves the
Poisson equation for the Euclidean ball.

Equation~\eqref{eq:poisson} can be derived from the generating function (see
\cite[(6.4.7)]{Andrews_Askey_Roy1999:special_functions})
\begin{multline*}
  \sum_{n=0}^\infty \frac{(\alpha+\beta+1)_n}{(\beta+1)_n}
  P_n^{(\alpha,\beta)}(\cos2\kappa\theta)r^n\\=
  \frac1{(1+r)^{\alpha+\beta+2}}
  \Hypergeom21{\frac{\alpha+\beta+1}2,\frac{\alpha+\beta+2}2}{\beta+1}%
  {\frac{4r\cos^2\kappa\theta}{(1+r)^2}}
\end{multline*}
by differentiating with respect to $r$ and applying contiguous relations for
the hypergeometric function.

We use the Poisson kernel to define the associated convolution operator acting
on continuous functions, for
which we use the same notation
\begin{align*}
  \mathcal{P}_r(f)(x)=
  \int_{X_{\alpha,\beta}}\mathcal{P}_r(\cos2\kappa\theta(x,y))f(y)
  \,d\sigma_{\alpha,\beta}(y).
\end{align*}
Furthermore, for $f\in\mathrm{C}(X_{\alpha,\beta})$ the limit relation
\begin{equation}\label{eq:limit}
  \lim_{r\to1-}\mathcal{P}_r(f)(x)=f(x)
\end{equation}
holds uniformly for $x\in X_{\alpha,\beta}$. The proof for this is completely
analogous to the proof that Fourier series of continuous functions converge in
the sense of Abel summation (see, for instance
\cite{Zygmund2002:trigonometric_series}).

If $F\in L^1(\mu_{\alpha,\beta})\cap\mathrm{C}([-1,1))$, then
\begin{equation}\label{eq:poissonF}
  \mathcal{P}_r(F)(\cos2\kappa\theta(x,z))=
  \sum_{n=0}^\infty \widehat{F}(n)r^nP_n^{\alpha,\beta}(\cos2\kappa\theta(x,z)).
\end{equation}
Here the operator is applied to the zonal function
$F(\cos2\kappa\theta(x,\cdot))$. In this case \eqref{eq:limit} still holds
uniformly on $\cos2\kappa\theta(x,y)\in[-1,1-\epsilon]$ for $\epsilon>0$.


%% file: PosDef.tex
\section{Positive definite functions}\label{sec:pos-def}
Positive definite kernels play an important role in potential theory, harmonic
analysis, and operator theory on metric spaces. Unfortunately, for singular
kernels the general theory is scarce. Even the standard definition does not
apply in this case, although it is not hard to properly adjust it. Below we
define positive definiteness for singular zonal kernels on the two-point
homogeneous spaces $ X_{\alpha,\beta}$, but the extension to the general case
is straightforward.

\begin{definition}
  A measurable function $F:[-1,1]\to\mathbb{R}$ is called
  \emph{positive definite} on $ X_{\alpha,\beta}$, if for all signed Borel measures $\mu$, for which
  $F(\cos2\kappa\theta(\cdot,y))\in L^1(|\mu|)$ for all $y\in
  X_{\alpha,\beta}$, the inequality 
  \begin{equation}\label{eq:pos-def}
    E_F(\mu)=\iint\limits_{X_{\alpha,\beta}\times X_{\alpha,\beta}}
    F(\cos2\kappa\theta(x,y))\,d\mu(x)\,d\mu(y)\geq0
  \end{equation}
  holds. If there is strict inequality in \eqref{eq:pos-def} for all
  $\mu\neq0$, the kernel is called \emph{strictly positive definite}. A
  kernel $K$ is called \emph{conditionally positive definite}, if
  \eqref{eq:pos-def} holds for all signed measures of zero total mass
  \textup{(}$\mu(X)=0$\textup{)}.
\end{definition}
\begin{remark}\label{rem:schur}
  This definition is an immediate generalisation of the definition given by
  Bochner in \cite{Bochner1941:hilbert_distances_positive}. This classical
  definition requires the matrix $(F(\cos2\kappa\theta(x_i,x_j)))_{i,j}$ to be
  positive semidefinite for all choices of $x_i\in X_{\alpha,\beta}$
  ($i=1,\ldots,N$). It is obvious that this is equivalent to the fact that
  \eqref{eq:pos-def} holds for all finitely supported measures, and weak-$*$
  density of such measures implies the equivalence of the two definitions for
  continuous kernels $F \in C[-1,1]$.
  
  The original definition forces the kernel $F$ to be defined for all pairs
  $x$, $y$ in the domain.
A very prominent example, namely the Riesz kernel
$\|x-y\|^{-s}$ defined for
$X\subset\mathbb{R}^d$, does not satisfy this requirement when
$x=y$. The definition above avoids this issue and applies to singular
kernels. We mention that for the case of $\mathbb{R}^d$ singular positive
kernels have been studied in
\cite{Phillips_Schmidt_Zhigljavsky2019:extension_schoenberg_theorem}.
\end{remark}

We now demonstrate that, just like in the continuous case
\cite{Barbosa_Menegatto2016:strictly_positive_definite}, positive definiteness
on $X_{\alpha,\beta}$ amounts to non-negativity of the coefficients in the Jacobi
expansion.

\begin{theorem}\label{thm:posdef}
  Let $F:[-1,1)\to\mathbb{R}$ be a function in
  $L^1(\mu_{\alpha,\beta}) \cap \cont([-1,1))$. Then $F$ is positive definite if
  and only if all of  its Fourier coefficients are non-negative.
\end{theorem}
\begin{proof}
  We first prove that positive definiteness implies  non-negativity of the
  Fourier coefficients. For this purpose we consider the signed measure
  $d\mu_z(x)=P_n^{(\alpha,\beta)}(\cos(2\kappa\theta(x,z)))\,d\sigma_{\alpha,\beta}(x)$
  for fixed $z$ and compute its energy
  \begin{align}\label{eq:energyjacobi}
    E_F(\mu_z)=\iint\limits_{X_{\alpha,\beta}\times X_{\alpha,\beta}}
    &F(\cos2\kappa\theta(x,y))
    P_n^{(\alpha,\beta)}(\cos(2\kappa\theta(x,z)))\\
 \nonumber   \times &P_n^{(\alpha,\beta)}(\cos(2\kappa\theta(y,z)))\,d\sigma_{\alpha,\beta}(x)\,
    d\sigma_{\alpha,\beta}(y).
  \end{align}
  Since $F$ is positive definite and this energy is obviously finite, we have $E_F(\mu_z) \ge 0$.  

  
    


  We shall use the following relation (see e.g. \cite[Lemma
  9.1]{Bilyk_Matzke_Nathe2024:geodesic_distance_riesz}), which can be seen as a
  projective analogue of the classical Funk--Hecke formula and is easily proved
  using the addition formula
  \begin{align}\label{eq.convo}
 \int_{X_{\alpha,\beta}}  P_n^{(\alpha,\beta)}& \big(\cos (2\kappa \vartheta (x,z)) \big)  P_n^{(\alpha,\beta)}  \big(\cos (2\kappa \vartheta (y,z)) \big) ) d\sigma_{\alpha,\beta}  (z ) \\ 
 \nonumber &=  \frac{P_n^{(\alpha,\beta)} (1) }{m_n^{(\alpha,\beta)}} \, P_n^{(\alpha,\beta)} \big(\cos (2\kappa \vartheta (x,y)) \big) . 
 \end{align}

According to  \eqref{eq:zonal-int} and  \eqref{eq:defcoef}, we see that
\begin{align*}
\widehat{F}(n)=\frac{m_n^{(\alpha,\beta)}}{\left(P_n^{(\alpha,\beta)}(1)\right)^2} &
  \int_{X_{\alpha,\beta}}  \int_{X_{\alpha,\beta}} F(\cos(2\kappa\theta(x,y))) \\
  & \times  P_n^{(\alpha,\beta)}(\cos(2\kappa\theta(x,y)))
    \,d\sigma_{\alpha,\beta}(x)   \,d\sigma_{\alpha,\beta}(y).
\end{align*}
The introduction of the outer  integral with respect to $d\sigma_{\alpha,\beta}(y)$ above is harmless since the inner integral is independent of $y \in X_{\alpha,\beta}$.  Invoking \eqref{eq.convo} and interchanging the order of integration (which is justified since $F$ is integrable with respect to $d\sigma_{\alpha,\beta}$ and $P_n^{(\alpha,\beta)}$ is bounded), we obtain
\begin{align}\label{eq:coef-energy}
\widehat{F}(n) =  \frac{\big( m_n^{(\alpha,\beta)}\big)^2}{\left(P_n^{(\alpha,\beta)}(1)\right)^3}  \int_{X_{\alpha,\beta}}  E_F (\mu_z)     \,d\sigma_{\alpha,\beta}(z) \ge 0,
\end{align}
which proves the first implication.

Conversely, assume now that all Fourier coefficients of $F$ are non-negative. If the series
\begin{equation*}
  \sum_{n=0}^\infty \widehat{F}(n)P_n^{(\alpha,\beta)}(1)
\end{equation*}
converges, then the Fourier series of $F$ converges absolutely and uniformly 
and thus $F$ is continuous, and the assertion follows from the classical result
for continuous kernels.  Thus we assume that the series diverges. Let $\mu$ be
a signed measure such that $F(\cos2\kappa\theta(\cdot,y))\in L^1(|\mu|)$ for
all $y\in X$. Such a measure is necessarily continuous (atomless) under our
assumption on $F$.

Then for $0<r<1$ the function $\mathcal{P}_r(F)$ is continuous and has an
absolutely and uniformly convergent Fourier series. Furthermore,
$\mathcal{P}_r(F)$ is positive definite. The corresponding energy
\begin{multline*}
  E_{\mathcal{P}_r(F)} (\mu) =\iint\limits_{X_{\alpha,\beta}\times X_{\alpha,\beta}}
  \mathcal{P}_r(F)(\cos2\kappa\theta(x,y))\,d\mu(x)\,d\mu(y)\\
  =
  \iint\limits_{X_{\alpha,\beta}\times X_{\alpha,\beta}}
  \int\limits_{X_{\alpha,\beta}}F(\cos2\kappa\theta(x,z))
  \mathcal{P}_r(\cos2\kappa\theta(z,y))\,d\sigma_{\alpha,\beta}(z)
  \,d\mu(x)\,d\mu(y)
\end{multline*}
is then positive for $0<r<1$ and is an increasing function of $r$, which can be seen from \eqref{eq:poissonF}, positivity of $\widehat{F}(n)$, and positive definiteness of $P_n^{(\alpha,\beta)}$. 
We interchange the order of integration and define
\begin{equation*}
  G(z)=\int_{X_{\alpha,\beta}}F(\cos2\kappa\theta(x,z))\,d\mu(x).
\end{equation*}
The function $G$ is then continuous under our assumptions: fix $z_0\in X$
and let $\eps>0$. Then there exists a $\delta>0$ such that
\begin{equation*}
  \forall z\in B(z_0,\delta)\quad\forall x\in X\setminus B(z_0,2\delta):
  |F(\cos2\kappa\theta(x,z_0))-F(\cos2\kappa\theta(x,z))|<\eps.
\end{equation*}
Since $\mu$ is continuous and $F(\cos2\kappa\theta(\cdot,y)) \in L^1(|\mu|)$
for all $z\in X$, we can choose $\delta$ so small that
\begin{equation*}
  \int_{B(z_0,2\delta)}|F(\cos2\kappa\theta(x,z))|\,d|\mu|(x)<\eps
\end{equation*}
for all $z\in X$. Then we have for $z\in B(z_0,\delta)$
\begin{multline*}
  |G(z)-G(z_0)|\leq
  \int\limits_{B(z_0,2\delta)}\left(|F(\cos2\kappa\theta(x,z))|+
    |F(\cos2\kappa\theta(x,z_0))|\right)\,d|\mu|(x)\\
  +\int\limits_{X_{\alpha,\beta}\setminus B(z_0,2\delta)}
  |F(\cos2\kappa\theta(x,z_0))-F(\cos2\kappa\theta(x,z))|\,d|\mu|(x)<
  (2+|\mu|(X_{\alpha,\beta}))\eps,
\end{multline*}
which shows that $G$ is continuous.

We then have
\begin{equation*}
  E_{\mathcal{P}_r(F)}(\mu)=\int\limits_{X_{\alpha,\beta}}
  \int\limits_{X_{\alpha,\beta}}
  G(z)\mathcal{P}_r(\cos2\kappa\theta(y,z))\,d\sigma_{\alpha,\beta}(z)\,d\mu(y).
\end{equation*}
By the continuity of $G$ the inner integral tends to $G(y)$ uniformly in $y$
for $r\to1$. From this we obtain
\begin{equation*}
  \lim_{r\to1}E_{\mathcal{P}_r(F)}(\mu)=\int\limits_{X_{\alpha,\beta}}G(y)\,d\mu(y)=
  E_F(\mu).
\end{equation*}
Since $E_{\mathcal{P}_r(F)}\geq0$ for $0<r<1$, we have obtained $E_F(\mu)\geq0$.
\end{proof}

The following theorem is an immediate corollary of our proof using the
observation that $E_{\mathcal{P}_r(F)}(\mu)$ is an increasing function of $r$.
\begin{theorem}\label{thm:strict}
  Let $F:[-1,1)\to\mathbb{R}$ be a function in
  $L^1(\mu_{\alpha,\beta})\cap \cont([-1,1))$. Then $F$ is strictly positive
  definite if and only if all its Fourier coefficients are strictly positive.
\end{theorem}

In addition, one can easily see that, since measures of mass zero annihilate constants, the proof easily adapts to yield a similar characterization holds for (strict) conditional positive definiteness. 
\begin{theorem}\label{thm:conditional}
  Let $F:[-1,1)\to\mathbb{R}$ be a function in
  $L^1(\mu_{\alpha,\beta})\cap \cont([-1,1))$. Then the kernel $F$ is (strictly) conditionally positive
  definite if and only if  its Fourier coefficients $\widehat{F}(n)$ are non-negative  (strictly positive) for all $n\ge 1$. 
\end{theorem}

\subsection*{Positive definiteness and energy minimization}

In general, the connection between positive definiteness and minimizing energy
integrals is classical
\cite{Borodachov_Hardin_Saff2019:discrete_energy_rectifiable}. However, general
equivalence results for singular kernels are not found in the literature. We
use characterizations proved above to yield the necessary and sufficient
condition for the energy $E_F (\mu)$ with a singular kernel $F$ to be minimized
by the normalized Haar measure $\sigma_{\alpha,\beta}$ amongst all Borel
probability measures.

\begin{theorem}\label{thm:energyminimize}
Let $F:[-1,1)\to\mathbb{R}$ be a function in $L^1(\mu_{\alpha,\beta})\cap \cont([-1,1))$. Then the energy integral
 \begin{equation}\label{eq:energy}
    E_F(\mu)=\iint\limits_{X_{\alpha,\beta}\times X_{\alpha,\beta}}
    F(\cos2\kappa\theta(x,y))\,d\mu(x)\,d\mu(y)
  \end{equation}
  is \textup{(}uniquely, resp.\textup{)} minimized by the normalized Haar
  measure $\sigma_{\alpha,\beta}$ amongst all Borel probability measures if and
  only if the kernel $F$ is \textup{(}strictly, resp.\textup{)} conditionally
  positive definite.
\end{theorem}

This result, according to Theorem \ref{thm:conditional}, implies that the fact
that the uniform measure $\sigma_{\alpha,\beta}$ minimizes $E_F (\mu)$
(i.e. energy minimization leads to uniform distribution) amounts to checking
the positivity of the Fourier (Jacobi) coefficients.  In the case of continuous
kernels, this result is well known and can be found in different variations in
\cite{Bilyk_Dai2019:geodesic_distance_riesz,
  Schoenberg1942:positive_definite_functions,
  Damelin_Grabner2003:energy_functional_numeric,
  Bochner1941:hilbert_distances_positive}, but it has not been stated in such
generality for singular kernels. Certain ad hoc approximation tricks have been
used for some particular (Riesz) kernels
\cite{Anderson_Dostert_Grabner+2023:riesz_green_energy,
  Damelin_Grabner2003:energy_functional_numeric,
  Bilyk_Dai2019:geodesic_distance_riesz}, and it was shown in
\cite[Theorem 4.2]{Bilyk_Matzke_Nathe2024:geodesic_distance_riesz} that
positivity of the coefficients is sufficient for a fairly general class of
singular kernels, with certain restrictions on the behavior at the
singularity. We demonstrate that positivity (non-negativity) of the
coefficients is a necessary and sufficient condition, assuming merely the
integrability of the kernel (i.e. finiteness of the energy
$E_F (\sigma_{\alpha,\beta})$) and continuity away from the singularity.

\begin{proof}[Proof of Theorem~\ref{thm:energyminimize}]
  Let us first assume that $F$ is conditionally positive definite. According to
  Theorem \ref{thm:conditional}, this is equivalent to the fact that
  $\widehat{F} (n) \ge 0$ for $n\ge 1$. Since $\mathcal{P}_r(F)$ is continuous
  and conditionally positive definite, the aforementioned classical results
  imply that $E_{\mathcal{P}_r(F)}(\mu)$ is minimized by
  $\sigma_{\alpha,\beta}$. Arguing exactly as in the second part of Theorem
  \ref{thm:posdef}, we find that for any Borel probability measure such that
  $E_F (\mu)$ is finite
\begin{equation}
 E_F(\mu) =  \lim_{r\to1}E_{\mathcal{P}_r(F)}(\mu) \ge  \lim_{r\to1}E_{\mathcal{P}_r(F)}( \sigma_{\alpha,\beta}) = E_{F}( \sigma_{\alpha,\beta}),
\end{equation}
i.e. $E_F$ is minimized by $\sigma$.  In the case of strict conditional
positive definiteness, uniqueness of the minimizer follows from general
potential theory, see
\cite[Theorem~4.2.7]{Borodachov_Hardin_Saff2019:discrete_energy_rectifiable}.

Conversely, if $F$ is not conditionally positive definite, i.e. there exists
$n\ge 1$ such that $\widehat{F} (n) <0$, one can choose the
measure
$$d\mu_z(x)=\big(1 + \varepsilon
P_n^{(\alpha,\beta)}(\cos(2\kappa\theta(x,z)))\big)
\,d\sigma_{\alpha,\beta}(x),$$ where $\varepsilon >0$ is chosen small enough so
that the measure is positive. It is standard to show, using the addition
formula for Jacobi polynomials (see
e.g. \cite[Theorem~2.8]{Anderson_Dostert_Grabner+2023:riesz_green_energy}) that
the value of the energy $E_F (\mu_z)$ is independent of
$z \in X_{\alpha,\beta}$. Therefore, a computation almost identical to the one
presented in the first part of the proof of Theorem \ref{thm:posdef} and based
on \eqref{eq.convo} and \eqref{eq:coef-energy} demonstrates that
\begin{equation}
E_F (\mu_z) =  E_{F}( \sigma_{\alpha,\beta}) +     \varepsilon^2   \frac{\left(P_n^{(\alpha,\beta)}(1)\right)^3}{\big( m_n^{(\alpha,\beta)}\big)^2} \widehat{F}(n) <  E_{F}( \sigma_{\alpha,\beta}),
\end{equation}
i.e. $ \sigma_{\alpha,\beta}$ does not minimize the energy. If $F$ is conditionally  positive definite, but not strictly, then   $ \widehat{F}(n) =0$ for some $n\ge 1$ and thus $E_F (\mu_z) =  E_{F}( \sigma_{\alpha,\beta})$, hence $\sigma_{\alpha,\beta}$ cannot be a unique minimizer. 
\end{proof}



%% file: Schur.tex
\section{Schur's lemma for singular integrable kernels}
\label{sec:schurs-lemma}
Schur's lemma is the most important structural result in the context of
continuous positive definite functions. It asserts that the Hadamard product of
two positive definite matrices is again positive definite. As a consequence the
product of two continuous positive definite functions is again positive
definite. This implication depends directly on the fact of continuity (see
Remark~\ref{rem:schur}). The following theorem gives an analogue of Schur's
lemma in the context of singular positive definite kernels.
\begin{theorem}\label{thm:schur-sing}
  Let $F_1,F_2:[-1,1)\to\mathbb{R}$ be positive definite functions in
  $L^1(\mu_{\alpha,\beta})\cap \cont([-1,1))$ such that their product is also in
  $L^1(\mu_{\alpha,\beta})\cap \cont([-1,1))$. Then $F_1F_2$ is also positive
  definite.
\end{theorem}
\begin{proof}
  We start with proving the theorem for the restricted case with
  $F_1\in L^1(\mu_{\alpha,\beta})\cap \cont([-1,1)$ and $F_2\in
  \cont([-1,1])$. We define
  \begin{equation*}
    F_{1,r}=\mathcal{P}_r(F_1)
  \end{equation*}
  which is continuous and positive definite on $[-1,1]$ for $0<r<1$.
  
  Both functions are positive definite by assumption. Let $\mu$ be a signed
  measure such that
  $F_1(\cos(2\kappa\theta\cdot,y))\in  L^1(|\mu|)$. Then we have
  \begin{align*}
    0&\leq\int_{X_{\alpha,\beta}}\int_{X_{\alpha,\beta}}
    F_{1,r}(\cos2\kappa\theta(x,y))F_2(\cos2\kappa\theta(x,y))
    \,d\mu(x)\,d\mu(y)\\
   &=
   \int_{X_{\alpha,\beta}}\int_{X_{\alpha,\beta}}\int_{X_{\alpha,\beta}}
   F_1(\cos2\kappa\theta(x,z))\mathcal{P}_r(\cos2\kappa\theta(z,y))\\
   &\quad\times F_2(\cos2\kappa\theta(x,y))\,d\sigma(z)\,d\mu(x)\,d\mu(y).
 \end{align*}
 We interchange the order of integration and set
 \begin{equation*}
   G(y,z)=\int_{X_{\alpha,\beta}}F_1(x,z)F_2(x,y)\,d\mu(x).
 \end{equation*}
 This function is then continuous on $X_{\alpha,\beta}\times X_{\alpha,\beta}$
 by the same argument as in the
 proof of Theorem~\ref{thm:posdef}. By the properties of the Poisson kernel
 \begin{equation*}
   \int_X G(y,z)\mathcal{P}_r(\cos2\kappa\theta(z,y))\,d\sigma(z)\to G(y,y)
 \end{equation*}
 for $r\to1$ uniformly on $X_{\alpha,\beta}$. From this we obtain
 \begin{align*}
   0&\leq \int_{X_{\alpha,\beta}}\int_{X_{\alpha,\beta}}
   F_{1,r}(\cos2\kappa\theta(x,y))F_2(\cos2\kappa\theta(x,y))
   \,d\mu(x)\,d\mu(y)\\
   &\to
   \int_{X_{\alpha,\beta}} G(y,y)\,d\sigma_{\alpha,\beta}(y)\\
   &=\int_X\int_X F_1(\cos2\kappa\theta(x,y))F_2(\cos2\kappa\theta(x,y))
   \,d\mu(x)\,d\mu(y)
 \end{align*}
 for $r\to1$.

 For proving the general case we apply the assertion proved above to $F_1$ and
 $F_{2,t}$ (defined analogous to $F_{1,t}$) and again take the limit  as $r\to1$.
\end{proof}


%% file: Schoenberg.tex
\section{Functions that are positive definite for all
  $\mathbb{FP}^{d-1}$}\label{sec:functions-for-all}
In \cite{Schoenberg1942:positive_definite_functions} Schoenberg characterizes
positive definite functions on the spheres $\mathbb{S}^{d-1}$ by the positivity
of the coefficients of their ultraspherical expansion. Furthermore, he gives a
characterization of the functions that are positive definite on all spheres. He
proves the following theorem:
\begin{theorem}[{\cite[Theorem~2]{Schoenberg1942:positive_definite_functions}}]
  \label{thm:schoenberg}
  A function $f:[-1,1]\to\mathbb{R}$ is positive definite on all spheres
  $\mathbb{S}^{d-1}$, if and only if it can be expressed in the form
  \begin{equation*}
    f(\cos\theta)=\sum_{n=0}^\infty a_n\cos^n\theta
  \end{equation*}
  with $a_n\geq0$ for all $n$ and $\sum_{n=0}^\infty a_n<\infty$. 
\end{theorem}

\begin{remark}
  We note here that Guella and Menegatto
  \cite{Guella_Menegatto2018:limit_formula_jacobi} proved a similar theorem for
  series in terms of Jacobi polynomials $P_n^{(\alpha_m,\beta_m)}$ with
  $\alpha_m\to\infty$ and $\frac{\beta_m}{\alpha_m}\to c$ for $m\to\infty$.
\end{remark}
In this section we will show the corresponding result for the projective spaces
$\mathbb{FP}^{d-1}$ and also extend it to singular functions.  It will be
convenient to use normalized Jacobi polynomials in this section
\begin{equation*}
  p_n^{(\alpha,\beta)}(t)=\frac{P_n^{(\alpha,\beta)}(t)}
    {P_n^{(\alpha,\beta)}(1)}.
\end{equation*}
We first state and prove a lemma that will be essential for the proof of this
 theorem.
\begin{lemma}\label{lem:uniform}
  For every $n\in\mathbb{N}_0$ the limit relation
  \begin{equation*}
    \lim_{\alpha\to\infty}p_n^{(\alpha,\beta)} (\cos2\theta) =\cos^{2n}\theta
  \end{equation*}
  holds uniformly for   $\theta\in[0,\frac\pi2]$.
\end{lemma}
\begin{proof}[Proof of Lemma~\ref{lem:uniform}]
  We fix $\beta$ and first notice that
  \begin{multline*}
    p_n^{(\alpha,\beta)}(\cos2\theta)=\sum_{m=0}^n(-1)^{n-m}\binom nm
    \frac{(\beta+m+1)(\beta+m+2)\cdots(\beta+n)}
    {(\alpha+1)(\alpha+2)\cdots(\alpha+n-m)}\\\times
    \left(\frac{1+\cos2\theta}2\right)^m
    \left(\frac{1-\cos2\theta}2\right)^{n-m}.
  \end{multline*}
  For fixed $n$ every term tends to $0$ uniformly in $\theta$ for
  $\alpha\to\infty$, except for the term with $m=n$, which equals
  $(\frac{1+\cos2\theta}2)^n=\cos^{2n}\theta$. This shows uniform convergence
  for $\theta\in[0,\frac\pi2]$ for all $n$.

\end{proof}

\begin{lemma}\label{lem:pn-bound}
  For $\beta\geq0$ and $\beta=-\frac12$ the normalized Jacobi polynomials
  satisfy the following inequalities
  \begin{equation}
    \label{eq:pn-bound1}
    |p_n^{(\alpha,\beta)}(\cos2\theta)|
    \leq\frac{\Gamma(\alpha+1)}{\Gamma(\beta+1)
      (n\sin^2\theta_0)^{\alpha-\beta}}
  \end{equation}
  valid for 
  $\theta_0\leq\theta\leq\frac\pi2$ and
\begin{equation}
    \label{eq:pn-bound2}
    |p_n^{(\alpha,\beta)}(\cos2\theta)|\leq\frac{\Gamma(\alpha+1)}
    {\Gamma(\alpha-\beta)}
    \frac{(\cos\theta_0)^{2n}}{(\alpha-\beta-1-n\tan^2\theta_0)^{\beta+1}}
  \end{equation}
  valid for $\theta_0\leq\theta\leq\frac\pi2$ and
  $n\tan^2\theta_0<\alpha-\beta-1$.
\end{lemma}
\begin{proof}[Proof of Lemma~\ref{lem:pn-bound}]
  For $\beta\geq0$ we use the formula (see
  \cite{Koornwinder1972:addition_formula_jacobi})
  \begin{equation}
    \label{eq:koornwinder}
    \begin{split}
      p_n^{(\alpha,\beta)}(\cos2\theta)&=
 \frac{2\Gamma(\alpha+1)}{\sqrt\pi\Gamma(\alpha-\beta)\Gamma(\beta+\frac12)}\\
      &\times\int\limits_0^1\int\limits_0^\pi
      \left(\cos^2\theta-r^2\sin^2\theta+ir\cos\phi\sin2\theta\right)^n\\
      &\times(1-r^2)^{\alpha-\beta-1}r^{2\beta+1}\sin^{2\beta}\phi\,d\phi\,dr.
    \end{split}
  \end{equation}
  Notice that
  \begin{align*}
    &\left|\cos^2\theta-r^2\sin^2\theta+ir\cos\phi\sin2\theta\right|^2\\
    =&
    (\cos^2\theta+r^2\sin^2\theta)^2-4r^2\sin^2\theta\cos^2\theta\sin^2\phi\\
    \leq&
    (\cos^2\theta+r^2\sin^2\theta)^2=(1-(1-r^2)\sin^2\theta)^2.
  \end{align*}
  For proving \eqref{eq:pn-bound1} we estimate
  \begin{align*}
    &|p_n^{(\alpha,\beta)}(\cos2\theta)|
    \leq \frac{2\Gamma(\alpha+1)}{\sqrt\pi\Gamma(\alpha-\beta)\Gamma(\beta+\frac12)}\\
    \times&\int\limits_0^1\int\limits_0^\pi
    e^{-n(1-r^2)\sin^2\theta_0}(1-r^2)^{\alpha-\beta-1}r^{2\beta+1}\sin^{2\beta}\phi\,d\phi\,dr\\
    \leq&\frac{2\Gamma(\alpha+1)}{\Gamma(\alpha-\beta)\Gamma(\beta+1)}
    \int_0^1e^{-n(1-r^2)\sin^2\theta_0}(1-r^2)^{\alpha-\beta-1}r\,dr
  \end{align*}
  In the last integral we substitute $1-r^2$ and then extend the resulting
  integral to $\infty$ to obtain \eqref{eq:pn-bound1}.

  For proving \eqref{eq:pn-bound2} we estimate
  \begin{align*}
    &|p_n^{(\alpha,\beta)}(\cos2\theta)|
    \leq (\cos\theta_0)^{2n}
    \frac{2\Gamma(\alpha+1)}{\sqrt\pi\Gamma(\alpha-\beta)\Gamma(\beta+\frac12)}\\
    \times&\int\limits_0^1\int\limits_0^\pi
    e^{-r^2(\alpha-\beta-1-n\tan^2\theta_0)}r^{2\beta+1}\sin^{2\beta}\phi\,d\phi\,dr\\
    \leq&\frac{2\Gamma(\alpha+1)}{\Gamma(\alpha-\beta)\Gamma(\beta+1)}
    \int_0^1e^{-r^2(\alpha-\beta-1-n\tan^2\theta_0)}r^{2\beta+1}\,dr.
  \end{align*}
  We extend the range of integration to $\infty$ to obtain \eqref{eq:pn-bound2}.

  For $\beta=-\frac12$ we use the representation
  \begin{multline*}
    p_n^{(\alpha,-\frac12)}(\cos2\theta)=
    \frac{\Gamma(\alpha+2)}{\sqrt\pi\Gamma(\alpha+\frac32)}\\
\times\int_0^1\Re\left(\cos^2\theta-x^2\sin^2\theta+
      2ix\cos\theta\sin\theta\right)^n(1-x^2)^{\alpha-\frac12}\,dx.
 \end{multline*}
 This can be obtained from \eqref{eq:koornwinder} by taking
 the limit $\beta\to-\frac12$.
\end{proof}

\begin{definition}
  The space of continuous functions with a \textup{(}moderate\textup{)}
  singularity at the right end point is denoted by
  \begin{equation*}
  \csing([-1,1])=\left\{f\in\cont([-1,1))\mid \exists s\in\mathbb{R}^+:
    (1-t)^sf(t)\in\cont([-1,1])\right\}.
\end{equation*}
\begin{remark}
  Notice that every $f\in\csing([-1,1])$ is contained in
  $L^1(\mu_{\alpha,\beta})$ for $\alpha$ large enough.
\end{remark}
\end{definition}
\begin{theorem}\label{thm:sing-schoen}
  A function $f\in\csing([-1,1])$ is positive definite for all spaces
  $\mathbb{FP}^{d-1}$ \textup{(}for
  $\mathbb{F}\in\{\mathbb{R},\mathbb{C},\mathbb{H}\}$\textup{)} with $d$ large
  enough to ensure integrability, if and only if there are non-negative $a_n$
  \textup{(}$n\in\mathbb{N}_0$\textup{)} such that
  \begin{equation}\label{eq:sing-schoen}
    f(\cos2\theta)=\sum_{n=0}^\infty a_n\cos^{2n}\theta.
  \end{equation}
\end{theorem}
\begin{proof}
  The sufficiency of the condition is obvious by the fact that
  \begin{equation*}
    \cos^2\theta=\frac1{\alpha+\beta+2}P_1^{(\alpha,\beta)}(\cos2\theta)+
    \frac{\beta+1}{\alpha+\beta+2}P_0^{(\alpha,\beta)}(\cos2\theta).
  \end{equation*}
  Then Schur's lemma shows that all powers $\cos^{2n}\theta$ are positive
  definite.

  The proof of the necessity will make use of the same smoothing procedure as
  the proofs of Theorems~\ref{thm:posdef} and~\ref{thm:schur-sing}, this time
  taking care of the dependence on $\alpha$. Throughout this proof we express
  the function $f$ by its normalized Fourier expansion
  \begin{equation*}
    f(t)=\sum_{n=0}^\infty a_n^{(\alpha,\beta)}p_n^{(\alpha,\beta)}(t).
  \end{equation*}

Let $f\in\csing([-1,1])$ . Then there exists $s\geq0$ such that
\begin{equation*}
  (1-t)^sf(t)=g(t)\in\mathrm{C}([-1,1]).
\end{equation*}
We can choose $s\in\mathbb{N}$ with this property. Assume that $|g(t)|\leq 1$
for all $t\in[-1,1]$, which we can do without loss of generality. Then we
have
\begin{multline*}
  |\mathcal{P}_r(f)(1)|\leq C_{\alpha,\beta}
  \int_0^{\frac\pi2}\mathcal{P}_r(\cos2\theta)(\sin\theta)^{2\alpha-2s+1}
  (\cos\theta)^{2\beta+1}\,d\theta\\=
  \frac{1-r}{(1+r)^{\alpha+\beta+2}}C_{\alpha,\beta}
  \int_0^{\frac\pi2}
  \Hypergeom21{\frac{\alpha+\beta+2}2,\frac{\alpha+\beta+3}2}{\beta+1}%
  {\frac{4r\cos^2\theta}{(1+r)^2}}(\sin\theta)^{2\alpha-2s+1}
  (\cos\theta)^{2\beta+1}\,d\theta\\=
  \frac{\Gamma(\alpha+\beta+2)\Gamma(\alpha-s+1)}
  {\Gamma(\alpha+1)\Gamma(\alpha+\beta-s+2)}
  \frac{1-r}{(1+r)^{\alpha+\beta+2}}
  \Hypergeom21{\frac{\alpha+\beta+2}2,\frac{\alpha+\beta+3}2}{\alpha+\beta-s+2}%
  {\frac{4r}{(1+r)^2}};
\end{multline*}
the last equation can be seen integrating term by term.

We now use two basic hypergeometric function identities (see
\cite[3.1.9]{Andrews_Askey_Roy1999:special_functions} and
\cite[p.~42]{Magnus_Oberhettinger_Soni1966:formulas_theorems_special}) to
obtain
\begin{align*}
  &\frac{1-r}{(1+r)^{\alpha+\beta+2}}
  \Hypergeom21{\frac{\alpha+\beta+2}2,\frac{\alpha+\beta+3}2}{\alpha+\beta-s+2}%
  {\frac{4r}{(1+r)^2}}\\=
  &(1-r)\Hypergeom21{\alpha+\beta+2,s+1}{\alpha+\beta+2-s}{r}\\=
  &(1-r)^{-2s}\Hypergeom21{-s,\alpha+\beta+1-2s}{\alpha+\beta+2-s}{r}.
\end{align*}
The last hypergeometric function is a polynomial, since $s$ was assumed to be
an integer. For $\alpha\geq s$ all coefficients of this polynomial are bounded
as functions of $\alpha$.

Summing up, we have obtained that
\begin{equation}\label{eq:Prf-upper}
  \mathcal{P}_r(f)(1)=
  \sum_{n=0}^\infty a_n^{(\alpha,\beta)}r^n
  \leq\frac M{(1-r)^{2s}}
\end{equation}
for $\alpha\geq2s$ and a constant $M$ only depending on $s$ and $\beta$.

We use now that the function $f$ is integrable and positive definite for all
$\alpha>\alpha_0$. This implies the non-negativity of all coefficients
$a_n^{(\alpha,\beta)}$ for all $\alpha>\alpha_0$. Inserting $r=1-\frac1n$ into
\eqref{eq:Prf-upper} gives the upper bound
\begin{equation}\label{eq:coeff-bound}
  a_n^{(\alpha,\beta)}\leq 2eMn^{2s}
\end{equation}
for all $\alpha>\alpha_0$. This shows that there is an increasing sequence
$(\alpha_k)_{k\in\mathbb{N}}$ tending to $\infty$, such that
\begin{equation*}
  \forall n\in\mathbb{N}_0: \lim_{k\to\infty}a_n^{(\alpha_k,\beta)}=b_n
\end{equation*}
exists. Notice that then $b_n\leq 2eM n^{2s}$ for all $n\geq1$.

It remains to show that
\begin{equation*}
  f(\cos2\theta)=\sum_{n=0}^\infty b_n(\cos\theta)^{2n}.
\end{equation*}
We fix $\theta_0$. For a given $\epsilon>0$ we choose $N$ so large that
\begin{equation*}
  \sum_{n=N+1}^\infty b_n(\cos\theta_0)^{2n}\leq
  2eM\sum_{n=N+1}^\infty n^{2s}(\cos\theta_0)^{2n}<\epsilon.
\end{equation*}
Then we choose $\alpha_0$ so large that for all $\alpha_k>\alpha_0$ we have
\begin{equation*}
  \left|\sum_{n=0}^Na_n^{(\alpha_k,\beta)}p_n^{(\alpha_k,\beta)}(\cos2\theta)-
  \sum_{n=0}^Nb_n(\cos\theta)^{2n}\right|<\epsilon
\end{equation*}
for $\theta\in[0,\frac\pi2]$.

Then we have
\begin{equation*}
  \left|f(\cos2\theta)-\sum_{n=0}^\infty b_n(\cos\theta)^{2n}\right|<
  2\epsilon+\sum_{n=N+1}^\infty a_n^{(\alpha_k,\beta)}
  |p_n^{(\alpha_k,\beta)}(\cos2\theta)|
\end{equation*}
for $\theta\in[\theta_0,\frac\pi2]$ and $\alpha_k>\alpha_0$. It remains to bound
this last sum. For this purpose we use the estimates \eqref{eq:pn-bound1} and
\eqref{eq:pn-bound2} depending on the summation index
\begin{align*}
  &\sum_{n=N+1}^\infty a_n^{(\alpha_k,\beta)}|p_n^{(\alpha_k,\beta)}(\cos2\theta)|\\
  \leq&\frac{2Me\Gamma(\alpha_k+1)}{\Gamma(\alpha_k-\beta)}
  \sum_{N<n<C\alpha_k}n^{2s}\frac{(\cos\theta_0)^{2n}}
  {(\alpha_k-\beta-1-n\tan^2\theta_0)^{\beta+1}}\\
  +&
  \frac{2Me\Gamma(\alpha_k+1)}{\Gamma(\beta+1)}
  \sum_{n>C\alpha_k}n^{2s-\alpha_k+\beta},
\end{align*}
where the constant $C$ satisfies $C\tan^2\theta_0<1$ and will be chosen later.
We estimate further by
\begin{align*}
  &\frac{2Me\Gamma(\alpha_k+1)}
  {\Gamma(\alpha_k-\beta)((1-C\tan^2\theta_0)\alpha_k-\beta-1)^{\beta+1}}
  \sum_{n=N+1}^\infty n^{2s}(\cos\theta_0)^{2n}\\
  +&
  \frac{2Me\Gamma(\alpha_k+1)}
  {(\alpha_k-\beta-2s)\Gamma(\beta+1)(\alpha_k-\beta-2s-1)}
  \left(C\alpha_k\right)^{2s-\alpha_k+\beta+1}.
\end{align*}
The first sum is bounded by $\epsilon$ times a function bounded in $\alpha_k$
(notice that the quotient of $\Gamma$-functions behaves like
$\alpha_k^{\beta+1}$, which cancels with the denominator). The second term can be
analyses using the Stirling approximation for the $\Gamma$-function. This gives
$(Ce\sin^2\theta_0)^{-\alpha_k}\alpha_k^{2s+\beta+\frac12}$. We choose $C$ so that
$Ce\sin^2\theta_0>1$ still satisfying $C\tan^2\theta_0<1$. Then the second sum
tends to $0$ for $\alpha_k\to\infty$.

Summing up, we have shown that
\begin{equation*}
  f(\cos2\theta)=\sum_{n=0}^\infty b_n(\cos\theta)^{2n}
\end{equation*}
uniformly on $[\theta_0,\frac\pi2]$.
\end{proof}
\begin{remark}
  The proof shows that actually the limit
  \begin{equation*}
    \lim_{\alpha\to\infty}a_n^{(\alpha,\beta)}
  \end{equation*}
  exists for all $n$, since the representation of $f$ as a power series in
  $\cos^2\theta$ is unique.
\end{remark}
\begin{remark}
  Theorem~\ref{thm:sing-schoen} contains the generalization of Schoenberg's
  Theorem~\ref{thm:schoenberg} to the projective spaces. Furthermore, the
  estimate \eqref{eq:pn-bound1} shows that for any $f\in\csing([-1,1])$ the
  Fourier series converges uniformly on $[\theta_0,\frac\pi2]$ for all $\alpha$
  large enough.
\end{remark}

%% file: Riesz.tex
\section{Riesz kernels for the geodesic and chordal
  distance}\label{sec:riesz-kern-geod}
The most prominent examples of singular kernels are the Riesz kernels
given by negative powers of a distance. These kernels are the subject of
investigation in the context of classical potential theory (see
\cite{Landkof1972:foundations_modern_potential}).

The spaces $X$ under consideration can be equipped with two ``natural''
metrics, which exhibit rather different behavior in this context: the geodesic
metric given by the angle $\theta(x,y)$ and the chordal metric
$\chi(x,y)=\sin(\kappa\theta(x,y))$. These metrics are, of course,
equivalent; the geodesic metric is the metric coming from the structure of $X$
as a Riemannian manifold, whereas the chordal metric comes from embedding the
space $X$ into a suitable Euclidean space (on the sphere, up to a constant factor, this is the Euclidean distance). 
It is this second property that
turns out to be crucial for positive definiteness of the associated Riesz
kernels. As we shall see below, (conditional) positive definiteness (and consequently energy minimization properties) behave very differently for these two metrics. 

Before further discussing these two special cases, we begin by discussing general Riesz-type kernels on metric spaces. 
In \cite{Schoenberg1938:metric_spaces_positive} I.~J.~Schoenberg investigated
the question, under which conditions a metric space $(X,d)$ can be isometrically embedded into
a Hilbert space. He found out that the correct condition is exactly  the conditional
positive definiteness of the function $-d(x,y)^2$. Furthermore, he found that a
function $K(x,y)$ is conditionally positive definite, if and only if
$e^{\lambda K(x,y)}$ is positive definite for all $\lambda>0$. Thus the metric
space $(X,d)$ can be embedded to Hilbert space if and only if the Gaussians $e^{-\lambda
  d(x,y)^2}$ are positive definite for all $\lambda>0$.

For $s \in \mathbb R$,  define the Riesz kernel on the metric space $(X,d)$ as $\operatorname{sgn} (s) d(x,y)^{-s}$ if $ s\neq 0$ and $-\log d (x,y)$ if $ s=0$. We show that conditional positive definiteness for one value of the parameter implies conditional positive definiteness of the Riesz kernel for a whole interval of parameters.

\begin{theorem}\label{thm:posdefinterval}
Let $(X,d)$ be a compact metric space of Hausdorff dimension $D>0$. Assume that for some $s_0 < 0$ the Riesz kernel is conditionally  positive definite, then the Riesz kernels are conditionally positive definite for all $s \in [s_0, D)$. 
\end{theorem}

\begin{proof}
Let us assume that $s_0 <0$, i.e.   $-d (x,y)^{-s_0}$ is
  conditionally positive definite and hence  $e^{-\lambda d(x,y)^{-s_0}}$ is positive
  definite for all $\lambda>0$. From this we derive immediately that
  \begin{equation*}
    -d(x,y)^{-\gamma s_0}=
    \frac1{\Gamma(-\gamma)}\int_0^\infty \left(e^{-\lambda d(x,y)^{-s_0}}-1\right)
    \lambda^{-\gamma-1}\,d\lambda
  \end{equation*}
  is conditionally positive definite for $0<\gamma<1$. Indeed, for a Borel measure $\mu$ on $X$  of total mass zero, one immediately sees that $E_F (\mu) \ge 0$ by interchanging the order of integration. This proves the statement of the theorem for $s_0< s < 0$. 
  
 Furthermore, for  $s=0$, the logarithmic kernel
$-\log d(x,y)$ is also conditionally positive definite by
\begin{equation*}
  -\log d(x,y)=\lim_{s\to0+}\frac{d(x,y)^{-s}-1}s.
\end{equation*}

We can treat the case $s>0$ similarly. We find   
  that
  \begin{equation*}
    d(x,y)^{-s}=\frac1{\Gamma(-2s/s_0)}\int_0^\infty e^{-\lambda d(x,y)^{-s_0}}
    \lambda^{-\frac{s}{s_0} -1}\,d\lambda
  \end{equation*}
  is positive definite for $0<s<D$, where $D$ is the dimension of the
  space. This upper bound for $s$ comes from the potential theoretic fact that
  $X$ does not carry a non-zero measure $\mu$ with finite Riesz energy for
  $s\geq D$.
  \end{proof}

  \begin{remark}\label{rem:log}
  If $X$ is a compact two-point homogeneous manifold  and, e.g.,  $d = \theta$ is  the geodesic distance, then the statement of Theorem \ref{thm:posdefinterval} continues to hold for   $s_0=0$, i.e. positive definiteness for $s>0$ can be derived directly from the logarithmic case. Indeed,  assume that $- \log d(x,y)$ is conditionally positive definite and observe that  
  \begin{equation}
  d(x,y)^{-s} = e^{s \big(- \log d(x,y) + C \big)} e^{-Cs} ,
  \end{equation}
where the constant $C$ is chosen so that the kernel $- \log d(x,y) + C $ is positive definite. 
While we cannot apply the aforementioned result of Schoenberg directly (since the kernels are singular), one can still see that the kernel on the right hand side is positive definite by expanding the exponential in power series and observing that $(- \log d(x,y) + C )^n $ is integrable with respect to  the uniform measure for all $n\ge 1$ and therefore positive definite  by the singular version of Schur's lemma (Theorem \ref{thm:schur-sing}). 
\end{remark}

\begin{remark}\label{rem:conj} We conjecture that Theorem \ref{thm:posdefinterval} should hold for any $s_0< D$ on arbitrary compact metric spaces, but we haven't been able to treat the case $s_0>0$. 
\end{remark}

We now return to the setting of  two-point homogeneous  compact manifolds and first discuss the Riesz kernels with respect to the chordal distance $\chi(x,y)$. It is well known that these spaces isometrically embed into Euclidean spaces, and therefore  $-\chi(x,y)^2$ is
  conditionally positive definite, i.e. the condition of the Theorem \ref{thm:posdefinterval} hold with $s_0 = -2$. Therefore we obtain the following corollary:
  
\begin{corollary}\label{cor:chordal}
  Let $X$ be a compact two-point homogeneous manifold with the chordal distance
  $\chi(x,y)$. Then the associated Riesz kernels are conditionally positive
  definite for all $s$ with $-2\le s <D$, where $D$ is the dimension of $X$ as
  a real manifold.
\end{corollary}
  
While this might be the most streamlined proof, this statement can be found in
the literature. In the case of the sphere it is classical
\cite{Bjoerck1956:distributions_positive_mass,
  Borodachov_Hardin_Saff2019:discrete_energy_rectifiable}, while the case of
projective spaces has been investigated in
\cite{Bilyk_Matzke_Nathe2024:geodesic_distance_riesz,
  Anderson_Dostert_Grabner+2023:riesz_green_energy}. It is also known that the
range is sharp: these kernels are not positive definite for $s<-2$.

The case of the geodesic distance is much more complicated due to the following
geometric reason. Since the spaces $X$ all have closed geodesics, there cannot
exist an embedding of the metric space $(X,\theta)$ into a Hilbert space. Thus
$-\theta(x,y)^2$ cannot be conditionally positive definite. Taking four equally
spaced points on a closed geodesic even shows that $-\theta(x,y)^\gamma$ is not
conditionally positive definite for $1<\gamma\leq2$. Beyond this range, the
behavior becomes very different for the sphere and projective spaces.

On the sphere $\mathbb{S}^{d-1}$ the kernel $-\theta(x,y)$ is conditionally
positive definite for any dimension $d\ge 2$ (see
\cite{Gangolli1967:positive_definite_kernels, Levy1965:processus_stochastiques,
  Skriganov2019:point_distributions_two-point}). This can be immediately
concluded from the easy direction of Theorem \ref{thm:schoenberg}, see
\cite{Schoenberg1942:positive_definite_functions}, since $\frac\pi2-\arccos(t)$
has only non-negative power series coefficients. Since $t$ is positive
definite, so are all of its powers are by Schur's lemma, and thus so is
$\frac\pi2-\arccos(t)$.

Therefore, applying Theorem \ref{thm:posdefinterval}, we recover the following
result \cite{Bilyk_Dai2019:geodesic_distance_riesz}:
  \begin{corollary}\label{cor:geodsphere}
  Let $X = \mathbb S^{d-1}$ be the sphere with the  metric  $\theta (x,y)$. Then the associated Riesz kernels are conditionally positive definite for all $s$ with $-1\le s <d-1$. 
\end{corollary}
  
In the case of projective spaces $\mathbb{FP}^{d-1}$ the situation is
drastically different (unless, of course, $d=2$, in which case these spaces are
isometric to the sphere
$\mathbb S^{\operatorname{dim}_{\mathbb R} \mathbb F}$). For $d\ge 2$, the
geodesic Riesz kernel with $s=-1$ is not conditionally positive definite, see
\cite{Gangolli1967:positive_definite_kernels} or \cite[Lemma
5.2]{Bilyk_Matzke_Nathe2024:geodesic_distance_riesz}. Energy minimization of
the geodesic Riesz kernel for $s=-1$ can be viewed as the continuous version of
the Fejes T\'oth conjecture on the sum of line angles in discrete geometry,
which has attracted considerable interest in the recent years, thus also
bringing attention to the geodesic distance Riesz energy on projective spaces
for all values of $s$ \cite{Lim_McCann2022:maximizing_expected_powers,
  Bilyk_Matzke_Nathe2024:geodesic_distance_riesz}.

We observe that a dimension independent statement akin to Corollary
\ref{cor:geodsphere} cannot hold on projective spaces.
Theorem~\ref{thm:sing-schoen} immediately implies that the Riesz kernel
$\theta(x,y)^{-s}$ cannot be conditionally positive definite for all
$\mathbb{FP}^{d-1}$ for any fixed $s>0$, because it cannot be expressed in the
form \eqref{eq:sing-schoen}. Indeed, functions of this form extend
differentiably to a function even around $\frac\pi2$, whereas $\theta^{-s}$
does not. (Theorem~\ref{thm:sing-schoen} treats positive definiteness, but the
case of conditional positive definiteness is immediate: we just need to require
that the coefficients $a_n$ are non-negative for $n\ge1$, i.e. $a_0$ could have
any sign).

Let $s_d(\mathbb{F})$ be the maximal number such that  for all
  $s<s_d(\mathbb{F})$
  the geodesic  Riesz kernel $\operatorname{sgn}(s) \theta^{-s}$ is not conditionally positive definite on $\mathbb{FP}^{d-1}$. 
Since (conditional) positive definiteness on ${\mathbb{FP}}^{d-1}$ obviously implies the same property in all projective spaces over $\mathbb F$ with smaller dimension, Theorem~\ref{thm:sing-schoen} and the discussion above shows that 
  \begin{equation*}
    \lim_{d\to\infty}s_d(\mathbb{F})=\infty.
  \end{equation*}
  Moreover, if $s_d(\mathbb{F}) \le  0$, then  the Riesz kernel is conditionally positive definite for all $s\in  \big( s_d(\mathbb{F}), d-1\big) $, according to  Theorem \ref{thm:posdefinterval} and the discussion thereafter. In addition, 
    if the conjecture in Remark \ref{rem:conj} is true, Riesz kernels should be conditionally positive definite for all $s>  s_d(\mathbb{F})$ independently of the sign of $s_d(\mathbb{F})$.

  The exact value of the critical exponent $s_d(\mathbb{F})$ which indicates the phase transition of  positive definiteness of the geodesic Riesz kernel  on the projective spaces remains absolutely mysterious. It follows from \cite{Bilyk_Matzke_Nathe2024:geodesic_distance_riesz} that $ s_d(\mathbb{R}) \le d-3$, $ s_d(\mathbb{C}) \le 2d-5$, $ s_d(\mathbb{H}) \le 4d-9$, $ s_3 (\mathbb{O}) \le 7$, since it has been shown that 
 $\theta^{-s}$ is positive definite for $\mathbb{FP}^{d-1}$ for $s $ above those values. 
  However, numerical experiments indicate that positive definiteness holds
  for a larger interval for $s$. In particular, numerical computations suggest that $ s_3(\mathbb{R}) \le -0.59$,  $s_4(\mathbb{R}) \le -0.125$, while it appears that $ s_d(\mathbb{R}) >0 $ for $d\ge 5$. Hence, whether the phase transitions happens in the negative or positive range of exponents seems to depend on the dimension. 
  
  In an attempt to gain a better insight into this phenomenon, we explored  the positive  definiteness of the logarithmic geodesic Riesz kernel $- \log  \theta (x,y)$ on $\mathbb{FP}^{d-1}$ corresponding to $s=0$ since it provides indication as to whether $s_d (\mathbb F)$ is positive or negative. We numerically computed the appropriate Jacobi coefficients $c_n$  of this kernel and  check their positivity. The results are presented in Table \ref{fig:1} and they seem to indicate that the critical value $s_d (\mathbb F)$ is positive in all cases other than $\mathbb{RP}^2$ and $\mathbb{RP}^3$.

\begin{table}
  \begin{tabular}[h]{|l|p{6cm}|}
    \hline
    space&definiteness\\\hline
    $\mathbb{RP}^2$&positive definite by
    \cite{Bilyk_Matzke_Nathe2024:geodesic_distance_riesz}\\\hline
    $\mathbb{RP}^3$&open, numerical computations suggest positive definiteness\\\hline
    $\mathbb{RP}^4$&not definite, numerical computations show that $c_8<0$\\\hline
    $\mathbb{CP}^2$&open, numerical computations suggest positive definiteness\\\hline
    $\mathbb{CP}^3$&not definite, numerical computations show that $c_6<0$\\\hline
    $\mathbb{HP}^2$&not definite, numerical computations show that $c_{10}<0$\\\hline
    $\mathbb{OP}^2$&not definite, numerical computations show that $c_8<0$\\\hline
  \end{tabular}
  \caption{Definiteness of the kernel $-\log\theta(x,y)$}\label{fig:1}
\end{table}
